    \documentclass[reqno]{amsproc}
    \usepackage{amsmath,amssymb,amsthm,amsfonts,latexsym}

    \theoremstyle{plain}
   \newtheorem{thm}{Theorem}
   \newtheorem{pro}[thm]{Proposition}
   \newtheorem{lem}[thm]{Lemma}
   \newtheorem{cor}[thm]{Corollary}

   \theoremstyle{definition}

   \theoremstyle{remark}
   \newtheorem{rem}[thm]{{\it Remark}}

\newcounter{dc}

\DeclareMathOperator{\okr}{{\stackrel{{\scriptscriptstyle{\mathsf{def}}}}{=}}}
\DeclareMathOperator{\D}{d\!}
 \DeclareMathOperator{\I}{i}
\DeclareMathOperator{\lin}{lin}
\DeclareMathOperator{\clolin}{clolin}

\def\dz#1{\mathcal D({#1})}
\def\funk#1#2#3{#1\colon#2\to#3}
\def\funkc#1#2#3#4#5{#1\colon#2\ni#3\mapsto#4\in#5}
\def\Ge{\geqslant}

\def\is#1#2{\langle#1,#2\rangle}

\def\Le{\leqslant}
\def\sbar#1{\,\overline{\!#1}}

\def\zb#1#2{\{{#1}\colon\ {#2}\}}
\def\ddc{\mathcal D}
\def\eec{\mathcal E}
\def\ffc{\mathcal F}

\def\hhc{\mathcal H}

\def\kkc{\mathcal K}
\def\llc{\mathcal L}

\def\ssf{\mathfrak S}

\def\ssfm{\mathfrak s}
\def\ttfm{\mathfrak t}
\def\uufm{\mathfrak u}

\def\fmax{\ffc_{max}}

\begin{document}

  \title[Framings and dilations]{Framings and dilations}
   \author[D.R. Larson]{David R. Larson}
   \address{Department of Mathematics, Texas A\&M University, College Station, TX}
   \email{larson@math.tamu.edu}
   \author[F.H. Szafraniec]{Franciszek Hugon Szafraniec}
   \address{Instytut Matematyki, Uniwersytet Jagiello\'nski,
ul. {\L}ojasiewicza 6, PL-30348 Krak\'ow}
   \email{umszafra@cyf-kr.edu.pl}
   \thanks{Research was initiated during Workshops in Analysis and Probability 2011 and 2012 at Texas A\&M University, College Station, TX. At its final stage that of the second author was supported by the MNiSzW grant NN201 1546438.}
   %
   \dedicatory{En comm\'emoration du centi\`eme anniversaire de B\'ela Sz\H{o}kefalvi Nagy,\\ le Grand Ma\^itre de la Th\'eorie des Op\'erateurs}
   \begin{abstract}
  The notion of framings,  recently emerging in \cite{caz} as generalization of the reconstraction formula generated by pairs of dual frames, is in this note extended substantially. This calls on refining the basic dilation results which still being in the flavor of {\em th\'eor\`eme principal} of B. Sz-Nagy \cite{app} go much beyond it.
        \end{abstract}
      \maketitle
 
 Framings considered here generalize those introduced in \cite{caz} though their environment  is more specific, the Hilbert space. On the other hand, they affect a space which may not be complete; an operator based procedure, Theorem \ref{2.24.11}, helps to develop further generations. Anyway the aforesaid non-completeness requires to extend know dilation theorems substantially; this is done in the second part of the paper.
 
     \section*{Framings}
      \subsection{Preliminaries}\label{2.30.12}
       Suppose $\mathcal H$ is a  complex Hilbert space.   A pair   $(g_{n})_{n=0}^{\infty}$ and $(h_{n})_{n=0}^{\infty}$ of sequences in $\hhc$, alternatively denoted as $(g_{n},h_{n})_{n=0}^{\infty}$, is said to be a {\em framing in} $\hhc$ if 
   \begin{equation}\label{2.01.09}
f=\sum_{n=0}^{\infty}\is{f}{g_{n}}h_{n}
\end{equation}
unconditionally for at least one nonzero $f$ in $\hhc$.  Because the right hand side of \eqref{2.01.09} does not depend on rearranging its order\,\footnote{\; In \cite[p.15]{lin} this important fact is mentioned.} framing is well defined.

Denote  by $\fmax$ the set of all $f$'s in $\hhc$ such that \eqref{2.01.09} holds; it always is a linear subspace of $\hhc$. If a choice of $\ffc\subset\fmax$ has been made we shortly say that  $(g_{n},h_{n})_{n=0}^{\infty}$ is a {\em framing for} $\ffc$\,\footnote{\;Framings were introduced in \cite{caz} and considered there even in the Banach space setting however $\ffc$  was chosen the whole space.} or, the other way,     $\ffc$  is a {\em framing space} for the framing$(g_{n},h_{n})_{n=0}^{\infty}$. While the choice of $\ffc$ is a matter of convenience,  $\fmax$ is uniquely determined. Notice that $\ffc_{\max}$ is stable under any (simultaneous) permutation of the framing $(g_{n},h_{n})_{n=0}^{\infty}$.

 The following easy to prove fact is worthy to be itemized.
\begin{pro}
Always
\begin{equation*}
 \fmax\subset\clolin(h_{n})_{n=0}^{\infty}.
\end{equation*}
\end{pro}

This has to be kept in mind when performing the construction proposed in Theorem \ref{2.24.11}.

%
%
%
%

Framing representation \eqref{2.01.09} is invariant with respect to the operation of rescaling. More precisely, if $(\alpha_{n})_{n=0}^{\infty}$ and $(\beta_{n})_{n=0}^{\infty}$ are such that 
\begin{equation}\label{1.27.11}
\alpha_{n}\sbar\beta_{n}=1 \text{ for all } n
\end{equation}
 then \eqref{2.01.09} is preserved for the new sequences $(\beta_{n}g_{n})_{n=0}^{\infty}$ and $(\alpha_{n}h_{n})_{n=0}^{\infty}$.  Example 1.3 in \cite{hlll} shows that this property is not pertinent to pairs of frames. Theorem \ref{2.24.11} below is a far going generalization of the rescaling procedure; in particular \eqref{1.27.11} is replaced by its operator version \eqref{1.17.01}.

\subsection{Generating framings}\label{3.30.12}   Instead of complex number sequences in the rescaling we use operators, possibly unbounded. This is a highly non-trivial generalization, also when one see rescaling as a very particular case of Proposition 3.8 in \cite{caz}.
\begin{lem}\label{2.24.11}
Let   $(g_{n},h_{n})_{n=0}^{\infty}$ be   a {framing} for $\ffc$. Suppose $A$ and $B$ are  closed operators, densely defined in $\hhc$ and such that\,\footnote{\;$\dz T$ stands for the domain of an operator $T$.}  $(g_{n})_{n=0}^{\infty}\subset \dz A$, $(h_{n})_{n=0}^{\infty}\subset \dz {B}$. 
Suppose, moreover, that there are given two linear subspaces of $\hhc$: 
\begin{equation*}
\ffc\neq\{0\}\text{ and }\ddc_{B}\subset\dz{B^{*}}\text{ dense in }\hhc.
\end{equation*}
\noindent If
\begin{equation}\label{1.11.06}
\ffc\subset\ffc_{A}\okr\zb g{g\in\dz{ A^{*}}\;\,\&\:\,A^{*}g\in\ffc_{\max}}
\end{equation}
and 
\begin{equation}\label{1.17.01}
 \is{A^{*}f}{B^{*}h}=\is fh,\quad f\in\ffc,\;h\in\ddc_B,
 \end{equation}
then for every $f\in\ffc$
\begin{equation}\label{7.01.09}
\is fh=\lim_{N}\is{\sum_{n}^{N}\is{A^{*}f}{g_{n}}Bh_{n}}{h},\quad h\in\ddc_{B}.
\end{equation}

\end{lem}    

\begin{proof} 

 For $f\in\dz{A^{*}}$ define the truncations
\begin{equation}\label{1.30.07}
f_{N}\okr\sum_{n}^{N}\is{f}{Ag_{n}}Bh_{n}.
\end{equation}
 With  \eqref{1.30.07} in mind we have 
 \begin{equation}\label{2.30.07}
\is{f_{N}}{h}=\sum_{n}^{N}\is{f}{Ag_{n}}\is{Bh_{n}}{h}=\sum_{n}^{N}\is{A^{*}f}{g_{n}}\is{h_{n}}{B^{*}h}= \is{\sum_{n}^{N}\is{A^{*}f}{g_{n}}h_{n}}{B^{*}h},\quad h\in\ddc_B.
\end{equation}
Now if $f\in\ffc$ the framing formula \eqref{2.01.09} is applicable. Thus $\sum_{n}^{N}\is{A^{*}f}{g_{n}}h_{n}$ tends to $A^{*}f$  and, due to \eqref{2.30.07},
\begin{equation}\label{x1.3.08}
\is{f_{N}}{h}\rightarrow\is{A^{*}f}{B^{*}h}\stackrel{\eqref{1.17.01}}=\is fh,\quad h\in\ddc_B.
\end{equation}
Therefore, using \eqref{x1.3.08}, still for {$f\in\ffc$}
\begin{align*}
\is{f}{h}=\lim_{N}\is{f_{N}}h=\lim_{N}\is{\sum_{n}^{N}\is{A^{*}f}{g_{n}}Bh_{n}}{h},\quad h\in\ddc_B.
\end{align*}
which is just \eqref{7.01.09}.
\end{proof}
\begin{rem}\label{1t.12.06}
It turns out to be important for the arguments  in the steps which follow to notice that the conclusion of Lemma \ref{2.24.11} is independent of any rearrangement of the series \eqref{2.01.09}, hence of any rearrangement of \eqref{8.13.05}.
\end{rem}

By a standard argument we arrive at the following.
\begin{cor}\label{1.13.05}
If the assumptions of Lemma \ref{2.24.11} are satisfied and the sums
\begin{equation}\label{3.13.05}
\sum_{n}^{N}\is{f}{Ag_{n}}Bh_{n},\quad N=0,1,\dots
\end{equation}
are norm bounded for $f\in\ffc_{A}$
then the series
\begin{equation*}
\sum_{n}^{\infty}\is{A^{*}f}{g_{n}}Bh_{n}
\end{equation*}
is weakly convergent to $f$ for every $f\in\ffc$.
\end{cor}

The definition of framing consists of two, somehow independent facts: the unconditional convergence of the right hand side of  \eqref{2.01.09} and the reconstruction formula \eqref{2.01.09} itself. While Lemma \ref{2.24.11}  causes the reconstruction formula \eqref{7.01.09}
to hold weakly, unconditional convergence requires a separate treatment. The relationship between these two in the context of framings is intriguing anyway.
\begin{thm}\label{2.13.05}
Let the assumptions of Lemma \ref{2.24.11} be satisfied. If any of the two following conditions 
\begin{enumerate}
\item[(i)] the series
\begin{equation}\label{8.13.05}
\sum_{n}^{\infty}\is{f}{Ag_{n}}Bh_{n},\quad  f\in\ffc
\end{equation}
is either weakly subseries convergent or strongly  subseries convergent;
\item[(ii)]  $B$ is a bounded operator
\end{enumerate}
is satisfied then 
 $(Ag_{n},Bh_{n})_{n=0}^{\infty}$ is a framing for $\ffc$.
\end{thm}
\begin{proof}
For (i) notice that weak convergence of  \eqref{3.13.05} itself makes its norm be bounded. Thus Corollary \ref{1.13.05} guarantees the series \eqref{8.13.05} is converging weakly to $f$. On the other hand, due to the assumption of (i) a direct application of Orlicz-Pettis theorem, cf.  \cite[Theorem 1.9]{dies} ensures the series $\sum_{n}^{\infty}\is{f}{Ag_{n}}Bh_{n}$ to converge unconditionally. Thus 
the only possibility is it to converge to $f$.

For (ii) consider any arrangement of the sequence $(\is{f}{Ag_{n}}Bh_{n})_{n=0}^{\infty}$ and keep the same notation for it. Then, due to the boundedness of $B$ and the reconstruction formula \eqref{2.01.09}, one has
$$
\|\sum\nolimits_{n}\is{f}{Ag_{n}}Bh_{n}\|\Le \|B\sum\nolimits_{n}\is{f}{Ag_{n}}h_{n}\|\Le
\|B\|\,\|\sum\nolimits_{n}\is{f}{Ag_{n}}h_{n}\|=\|B\|\,\|\sum\nolimits_{n}\is{A^{*}f}{g_{n}}h_{n}\|
$$
which, when applied to Cauchy fragments of the series in question, makes  unconditional convergence of the right hand side of  
\eqref{7.01.09}. This leads to the final conclusion while invoking Lemma \ref{2.24.11}
.
\end{proof}

\begin{rem}\label{2.17.01}
Theorem \ref{2.24.11} is symmetric in a sense that the assumptions are invariant under the replacement of $A$ and $B$. The only thing which may change, besides the framing formula \ref{1.30.07},
 is $\ffc_{A}$  turns into $\ffc_{B}$ which may be different. This way one gets a ``dual'' framing.
 
 If $\fmax=\hhc$, then symmetricity of framing is is the matter of Lemma 3.12 of \cite{hlll}
\end{rem}


Notice that $\ffc_{A}$ does not depend on $B$, this fact is related intimately to \eqref{1.17.01}.
Moreover if $\ffc_{\max}=\hhc$ then $\ffc_{A}=\dz{A^{*}}$. May it be  {\em sine qua non}?

\subsection{Example}\label{4.30.12}
By way of illustration execute the procedure provided by Theorem \ref{2.13.05}, part (ii)  using the simplest possible example of an unbounded operator. More precisely, let $\hhc=\llc^{2}[0,1]$ and $A$ be defined by
    \begin{gather*}
   \dz{A}  =
\zb{f \in L^2[0,1]} {f \text{ is absolutely
continuous in  [0,1]},\; f' \in L^2[0,1] 
},
   \\
   Af = - \I f', \; f \in \dz{A}.
   \end{gather*}
   Its adjoint $A^{*}$ is (cf. \cite[Section 119]{riesz})
    \begin{gather*}
   \dz{A^{*}}  =
\zb{f \in L^2[0,1]} {f \text{ is absolutely
continuous in [0,1]},\; f' \in L^2[0,1] \text { and } f(0) =
f(1)=0},
   \\
   A^{*}f = - \I f', \; f \in \dz{A}.
   \end{gather*}
   This advises us to take as $B$ the Volterra operator
   \begin{equation*}
   \I (Bf)(x)\okr\int_{0}^{x}f(t)\D t, \quad f\in\llc^{2}[0,1],\;x\in[0,1].
   \end{equation*}
  Then \eqref{1.17.01} is satisfied. Therefore \eqref{7.01.09} looks like
  \begin{equation*}
 f=\sum_{n=0}^{\infty}\is{f}{g_{n}^{'}}\int_{0}^{\,\boldsymbol{\cdot}}h_{n}(t)\D t,\quad f\in\mathcal F
  \end{equation*}
 provided  $(g_{n})_{n=0}^{\infty}\subset \dz A$ (watch ``dot'' in the upper limit of the integral put there for making the formula formal). 

\subsection{Operator valued measures associated to framings}\label{5x.30.12}
 Defining an operator\,\footnote{\;Sometimes it is denoted by $g\otimes h$.} $r_{g.h}$, with $g,h\in\hhc$, as
$$
r_{g,h}f\okr\is fg h,\quad f\in\ffc_{\max}
$$
we get a rank $1$ operator from $\ffc_{\max}$ to $\hhc$.
Given a framing
$(g_{n},h_{n})_{n=0}^{\infty}$, due to unconditional convergence of the series \eqref{2.01.09} any its subsequence is convergent and one can define an operator valued measure $F$ by
\begin{equation}\label{2.16.12}
F(\sigma)\okr\sum_{i\in\sigma}r_{g_{i},h_{i}},\quad \sigma\subset\{0,1,\ldots\}
\end{equation}
with the sum, if infinite, converging pointwisely in an unconditional way. $F(\sigma)$'s are linear operators in $\hhc$ with domains $\dz{F(\sigma)}=\ffc_{\max}$. Moreover $F(\{0,1,\ldots\})\subset I$ (notice the inclusion in the most interesting cases may be proper).

\begin{pro}\label{1.19.05}
Suppose \eqref{1.17.01} holds. If $F$ is the operator valued measure associated with the framing $(g_{n},h_{n})_{n=0}^{\infty}$ and $F_{A,B}$ that associated to the framing $(Ag_{n},Bh_{n})_{n=0}^{\infty}$ then for all $\sigma$'s
\begin{equation}\label{2.19.05}
F_{A,B}(\sigma)=BF(\sigma)A^{*}\;\text{ on }\;\ffc.
\end{equation}
with $\ffc$ as in \eqref{1.11.06}
\end{pro}
\begin{proof}
 It is convenient to check the equality \eqref{2.19.05} in the weak form. For finite $\sigma$'s it goes as follows:
 take $f\in\ffc$ and $h\in\hhc$, then, by \eqref{1.17.01},
 \begin{align*}
 \is{F_{A,B}(\sigma)f}h&=\is{\sum_{n\in\sigma}{\is f{Ag_{n}}}{Bh_{n}}}h=\is{B\sum_{n\in\sigma}{\is {A^{*}f}{g_{n}}}{h_{n}}}h=\is{BF(\sigma)A^{*}f}h.
 \end{align*}
If $\sigma$ is infinite  a combination of Orlicz-Pettis (or rather Pettis in this case) Theorem and closedness of $B$ completes the proof.
\end{proof}

 Proposition \ref{1.19.05} makes it unlikely in general the new framing becomes from the old one after rescaling.

%
%

\section*{General dilations}\label{66.31.12} 
%

 \subsection{General dilation {\em \`a la} B. Sz.--Nagy: algebraic part}\label{sto}
 
 Let $\mathfrak S$ be a semigroup\,\footnote{\,It is customary  to write the semigroup action  in the multiplicative way unless the semigroup is commutative, in  this case the additive notation is used.} with unit $1$, and let $\eec$ and $\ffc$ be two linear spaces. Given\,\footnote{\;$\bf L$ indicates linearity of the objects involved.} $\funk\varphi{\ssf} {\bf L(\ffc,\eec)}$, for $\mathfrak s\in \mathfrak S$ and $f\in \ffc$  the function $\varphi_{\ssfm,f}$  defined by 
 \begin{equation}\label{13.30.12}
 \varphi_{\ssfm,f}(\ttfm)\okr \varphi(\ttfm\ssfm)f,\quad \ttfm\in\mathfrak S
 \end{equation}
  which is a mapping of $\ssf$ into $\eec$; in other words $\varphi_{\ssfm,x}$ are members of $\eec^{\ssf}$.
 
   For any finite choice of $\xi_{i}$'s, $f_{i}$'s and $\ssfm_{i}$'s
\begin{equation*}
\sum_{i}\xi_{i}\varphi_{\ssfm_{i},f_{i}}=0\; \Longrightarrow\; \sum_{i}\xi_{i}\varphi_{\uufm\ssfm_{i},f_{i}}=0.
\end{equation*}
Indeed, $\sum_{i}\xi_{i}\varphi_{\ssfm_{i},f_{i}}=0$ means $\sum_{i}\xi_{i}\varphi_{\ssfm_{i},f_{i}}(\ttfm)=\sum_{i}\xi_{i}\varphi({\ttfm\ssfm_{i})f_{i}}=0$ for every $\ttfm\in\mathfrak S$. Therefore  $\sum_{i}\xi_{i}\varphi({\ttfm\uufm\ssfm_{i})f_{i}}=0$ for every $\ttfm\in\mathfrak S$ as well, this reads in turn as $\sum_{i}\xi_{i}\varphi_{\uufm\ssfm_{i},f_{i}}=0$. Consequently, $\varPhi(\uufm)$ acting as  
\begin{equation}\label{1.03.08}
\varPhi(\uufm)\varphi_{\ssfm,f}\okr \varphi_{\uufm\ssfm,f},\quad \ssfm\in\mathfrak S,\,f\in \ffc,
\end{equation}
is well defined as a linear operator on $\ddc$; the latter stands for the linear span $\lin\{\varphi_{\ssfm,f}\}_{\ssfm\ssf,f\in\ffc}$ considered in  $\eec^{\ssf}$. Notice that because $\varphi_{\ssfm,f}$ is linear in $f$ the space $\ddc$ is composed of the finite sums of $\varphi_{\ssfm,f}$'s.

The  operator $\funk T {\ffc} {\ddc}$ defined by 
\begin{equation}\label{1.27.12}
Tf\okr\varphi_{1,f}
\end{equation}
 is linear. On the other hand extending the definition 
 \begin{equation}\label{2.27.12}
 S\varphi_{\ssfm,f}\okr\varphi(\ssfm)f
 \end{equation}
  by linearity requires some argument. In order it to work we need to know that
$$
\sum_{i}\xi_{i}\varphi_{\ssfm_{i},f_{i}}=0\; \Longrightarrow\; \sum_{i}\xi_{i}\varphi(\ssfm_{i})f_{i}=0.
$$
Indeed, $\sum_{i}\xi_{i}\varphi_{\ssfm_{i},f_{i}}=0$ implies $\sum_{i}\xi_{i}\varphi(\ssfm_{i})f_{i}=\sum_{i}\xi_{i}\varphi_{\ssfm_{i},f_{i}}(1)=0$. Thus $ S\in{\bf L}(\ddc,\eec)$.

All this leads us to the following result.
\begin{thm}\label{1.08.08}
Suppose $\ssf$ is a semigroup with a unit, $\ffc$ a linear space. Referring to the construction done above we have
\begin{enumerate}
\item [(a)] the mapping, cf. \eqref{1.03.08},
$$\funkc  \varPhi {\mathfrak S} \uufm{\varPhi(\uufm)}{\bf L(\ddc)}$$
is a  unital semigroup homomorphism, in particular if $\ssf$ is commutative then so is $\zb{\varPhi(\uufm)}{\uufm\in\ssf}$, the range of $\varPhi$;
\item[(b)] the dilation formula
\begin{equation}\label{2.08.08}
\varphi(\uufm)=S\varPhi(\uufm)T,\quad \uufm\in\ssf
\end{equation}
holds;
\item[(c)] the minimality condition
\begin{equation}\label{3.08.08}
\ddc=\lin\zb{\varPhi(\uufm)Tf}{\uufm\in\ssf,\, f\in\ffc}
\end{equation}
holds as well;
\item[(d)] if, in addition, $\ssf$ is an algebra and $\varphi$ is linear then $\varPhi$ must necessarily be an algebra homomorphism.
\end{enumerate}
\end{thm}
\begin{proof}
Use \eqref{1.03.08} directly to check $\varPhi$ is a semigroup homomorphism. To prove \eqref{2.08.08} use \eqref{1.03.08}, definitions of $T$ and $S$, and write
$$S\varPhi(\uufm)Tf\stackrel{\eqref{1.27.12}}=S\varPhi(\uufm)\varphi_{1,f}\stackrel{\eqref{1.03.08}}=S\varphi_{\uufm,f}\stackrel{\eqref{2.27.12}}=\varphi_{\uufm,
f}.$$
Condition \eqref{3.08.08} is an immediate consequence of the definition of $\ddc$ as $\varPhi(\uufm)\varphi_{1.f}=\varphi_{\uufm,f}$.

If $\varphi$ is linear so is $\ssfm\mapsto\varphi_{\ssfm,f}$ and, consequently so is $\varPhi$. This establishes (d).
\end{proof}


Call any triplet $(\varPhi,S,T)$, or sometimes $\varPhi$ itself, satisfying \eqref{2.08.08} in condition (a) and (b) of Theorem \ref{1.08.08} a {\em dilation} of $\varphi$; it is called {\em minimal} if \eqref{3.08.08} in (c) is satisfied. 


Theorem \ref{1.08.08}   extract the algebraic component of Th\'eor\`eme Principal of \cite{app} and generalize Proposition 4.1 of \cite{hlll} by the way.

	\subsection{General dilation {\em \`a la} B. Sz.--Nagy: topological part}\label{7.30.12} 
	Now introduce some, rather simple, topology in $\ddc$ into the game. If $\eec$ is a topological linear space, then, because $\ddc\subset\eec^{\ssf}$ the topology in $\eec$ determines that of Tikhonov  in $\ddc$; in particular if $\eec$ is locally convex so is $\ddc$. The Tikhonov topology makes the dilation operators automatically continuous.

\begin{pro}\label{1.29.12} 
Suppose $\eec$ is a topolgical linear space. If $\varPhi$ is a minimal dilation of $\varphi$ constructed according to the recipe above, then   each operator $\varPhi(\uufm)$, $\uufm\in\ssf$, is continuous in the Tikhonov topology of $\ddc$. 
\end{pro}
\begin{proof}
Write down explicitly the topology involved and use the appropriate definitions. That is all.
\end{proof}

\begin{thm}\label{3.18.12}
Suppose $\eec$ is a topological linear space and $\varPhi$ is a minimal  dilation of $\varphi$. If $(\ssfm_{\alpha})_{\alpha}$ is net such that for any $\ttfm,\uufm\in\ssf$
\begin{equation*}
\varphi(\ttfm\ssfm_{\alpha}\uufm)f\rightarrow\varphi(\ttfm\ssfm\uufm)f,\quad f\in\ffc
\end{equation*}
then 
\begin{equation*}
\varPhi(\ssfm_{\alpha})w\rightarrow\varPhi(\ssfm)w,\quad w\in \ddc
\end{equation*}
where convergence is that in the Tikhonov topology inherited by $\ddc$.
\end{thm}
\begin{proof}
Using \eqref{1.03.08} and \eqref{13.30.12} there and back
we can write
$$
\varPhi(\ssfm_{\alpha})\varphi_{\uufm,f}(\ttfm)=\varphi_{\ssfm_{\alpha}\uufm,f}(\ttfm)=\varphi(\ttfm\ssfm_{\alpha}\uufm)f\rightarrow\varphi(\ttfm\ssfm\uufm)f=\varphi_{\ssfm\uufm,f}(\ttfm)=\varPhi(\ssfm)\varphi_{\ssfm,f}(\ttfm)
$$
which makes the conclusion.
\end{proof}
Theorem \ref{3.18.12} is in tune with part 3) of Th\'eor\`eme Principal of \cite{app} fitting in our general so far situation.

\subsection{Na\u{\i}mark-like dilation}\label{8.30.12}
All this done so far enables us to propose a Na\u{\i}mark-like dilation by specifying more $\ssf$. 
\begin{cor}\label{2a.08.08}
Let $\ssf$ be a family of subsets of a set $X$, which is closed for intersections, with $X\in\ssf$. Then the range of any multiplicative dilation $\varPhi$ is composed of commuting idempotents. 

Moreover, if $\ssf$ is a $\sigma$--algebra of subsets of $X$ and $\varphi$ is $\sigma$--additive in the strong  topology of $\eec$ then $\varPhi$ is $\sigma$--additive in the Tikhonov topology.
\end{cor}
\begin{proof}
Consider $\ssf$ as a semigroup with the semigroup multiplication being the intersection of sets; notice $X$ is a unit of $\ssf$. Then $\varPhi(\uufm)$'s are idempotents and, because $\ssf$ is commutative, they commute.

The second part of Corollary come from Theorem \ref{3.18.12}.
\end{proof}

For a first aid on unbounded idempotents we suggest \cite{ota}.

\section*{Positive definite dilations}
\subsection{General, not necessarily bounded  case}\label{9.30.12}
Suppose now  $\eec$ is an inner product space and $\ffc$ is a subspace of $\eec$. Suppose, moreover, $\ssf$ is a $\ast$--semigroups (or an involution semigroup in other words), with unit of course. Under these circumstances one can think of positive definiteness of $\varphi$: we say $\varphi$ is {\em positive definite} if
\begin{equation*}
\sum_{i,j}\is{\varphi(\ssfm^{*}_{j}\ssfm_{i})f_{i}}{f_{j}}_{\eec}\Ge0
\end{equation*}
for any finite choice of $(s_{k})_{k}\subset\ssf$ and $(f_{k})_{k}\subset\ffc$. Two important consequences are at hand
\begin{equation}
\is{\varphi(\ssfm^{*})f}g_{\eec}=\is f{\varphi(\ssfm)g}_{\eec},\quad \ssfm\in\ssf,\;f,g\in\ffc,\label{1.10.01}
\end{equation}
\begin{multline}
\big|\sum_{i,k}\is{\varphi(\ttfm^{*}_{k}\ssfm_{i})f_{i}}{g_{k}}_{\eec}\big|^{2} \Le
\sum_{i,j}\is{\varphi(\ssfm^{*}_{j}\ssfm_{i})f_{i}}{f_{j}}_{\eec}\sum_{k,l} \is{\varphi(\ttfm^{*}_{l}\ttfm_{k})g_{k}}{g_{l}}_{\eec},\\ (\ssfm_{i},\ttfm_{k})_{i,k}\subset\ssf,\; (f_{i},g_{k})_{i,k}\subset\ffc. \label{2.07.01}
\end{multline}
Let us try to set a (provisional so far) definition of an inner product in $\ddc$ (cf. Subsection 
\ref{sto}) by
\begin{equation}\label{1.07.01}
\is{\varphi_{\ssfm,f}}{\varphi_{\ttfm,g}}_{\ddc}\okr\is{\varphi(\ttfm^{*}\ssfm)f}g_{\eec}
\end{equation}
The Schwarz inequality \eqref{2.07.01} allows to extend the definition \eqref{1.07.01} by linearity with respect to the first variable as 
\begin{equation}\label{1x.07.01}
\is{\sum_{i}\varphi_{\ssfm_{i},f_{i}}}{\varphi_{\ttfm,g}}_{\ddc}\okr\sum_{i}\is{\varphi(\ttfm^{*}\ssfm_{i})f_{i}}g_{\eec}.
\end{equation}
Indeed, equal in \eqref{2.07.01} all the $g_{k}$'s 0 except one. Assuming $\sum_{i}\varphi_{\ssfm_{i},f_{i}}=0$ makes the extension possible. Act with respect to the second variable likewise getting $\is{\,\cdot\,}{\,-\,}_{\ddc}$ to be a semi-inner product. To prove it is an inner product notice that \eqref{1x.07.01} can be continued as follows. Take $\sum_{i}\varphi_{\ssfm_{i},f_{i}}$ such that $\is{\sum_{i}\varphi_{\ssfm_{i},f_{i}}}{\sum_{i}\varphi_{\ssfm_{i},f_{i}}}_{\ddc}=0$, apply the Schwarz inequality \eqref{2.07.01} to the right hand side of \eqref{1.07.01} so as to get
\begin{equation*}
\is{\sum_{i}\varphi_{\ssfm_{i},f_{i}}}{\varphi_{\ttfm,g}}_{\ddc}=0
\end{equation*}
for any $\ttfm$ and $g$. Because the right hand side of  \eqref{1.07.01} is  $\sum_{i}\varphi_{s_{i},f_{i}}(\ttfm^{*})$ we get it to be $0$ for every $t^{*}$ which completes the argument.

Now it is time to think of a completion of $\ddc$. Whatever the way to achieve it is denote the resulting space by $\hhc$.

Referring to Theorem \ref{1.08.08} the above can be summarized in the following.
\begin{cor}\label{2.10.01}
Suppose $\ssf$ is a $\ast$--semigroup with a unit, $\ffc$ a linear space, $\eec$ an inner product space and $\varphi$ is positive definite. Then $\ddc$ becomes an inner product space with the inner product extending sesquilineary that given by \eqref{1.07.01}. In addition to conditions {\em (a)\,--\,(d)} of \,{\em Theorem \ref{1.08.08}} the mapping $\varPhi$ is an $\ast$--semigroup homomorphism, that is
\begin{equation}\label{3.01.13}
\is{\varPhi(\uufm^{*})\varphi}{\psi}_{\ddc}=\is{\varphi}{\varPhi(\uufm)\psi}_{\ddc},\quad \varphi,\psi\in\ddc.
\end{equation}
\end{cor}
\begin{proof}
The only thing which requires some proof is \eqref{3.01.13}. Invoking  \eqref{1.10.01} it goes as as follows
\begin{align*}
\is{\varPhi(\uufm^{*})\varphi_{\ssfm,f}}{\varphi_{\ttfm,g}}_{\ddc}&\stackrel{\eqref{1.03.08}}=\is{\varphi_{\uufm^{*}\ssfm,f}}{\varphi_{\ttfm,g}}_{\ddc}\!\stackrel{\eqref{1.07.01}}
=\is{\varphi(\ttfm^{*}\uufm^{*}\ssfm)f)}{g}_{\ddc}
\\
&\stackrel{\eqref{1.10.01}}=\overline{\is{\varphi(\ssfm^{*}\uufm\ttfm)g}f}_{\eec}\stackrel{\eqref{1.07.01}}=\overline{\is{\varphi_{\uufm\ttfm,g}}{\varphi_{\ssfm,f}}}_{\ddc}\stackrel{\eqref{1.03.08}}=\overline{\is{\varPhi(\uufm)\varphi_{\ttfm,g}}{\varphi_{\ssfm,f}}}_{\ddc}\\
&=\is{\varphi_{\ssfm,f}}{\varPhi(\uufm)\varphi_{\ttfm,g}}.\qedhere
\end{align*}
\end{proof}
Corollary \ref{2.10.01} is the main step in establishing Sz.-Nagy's  Th\'eor\`eme Principal with the major ingredients of its proof being reorganized. It is in fact what is needed for dilations (extensions) of unbounded operators with invariant domain, cf. \cite{nagy}. The next step, boundedness of the dilating operators is discussed below.
\begin{rem}
\label{1.01.01}
Notice that
$$
\is{S\varphi_{\ssfm,f}}g_{\eec}\stackrel{\eqref{2.27.12}}=\is{\varphi(\ssfm)f)}g_{\eec}\stackrel{\eqref{1.07.01}}=\is{\varphi_{\ssfm,f}}{\varphi_{1,g}}_{\ddc}\stackrel{\eqref{1.27.12}}=\is{\varphi_{\ssfm,f}}{Tg}_{\ddc}
$$
which means $S$ and $T$ are adjoint in $\ddc$ each to the other or formally adjoint in $\kkc$.

If $\is{\varphi(1)f}f_{\eec}=\|f\|^{2}_{\eec}$, $f\in\eec$, then $\|Tf\|_{\ddc}=\|\varphi_{1,f}\|_{\ddc}=\|f\|_{\eec}$ for all $f$'s which means $T$ is an isometry. If, instead, $\is{\varphi(1)f}f_{\eec}\Le c\|f\|^{2}_{\eec}$, $T$ is just bounded.
\end{rem}

\begin{rem}\label{1.09.08}
At this stage a typical  uniqueness assertion for minimal dilations can be proved in a standard way.
  \end{rem}

\subsection{Boundedness}\label{10.30.12}
For $\uufm\in\ssf$ consider two conditions
\begin{equation}
\sum_{i,j}\is{\varphi(\ssfm^{*}_{j}\uufm^{*}\uufm\ssfm_{i})f_{i}}{f_{j}}_{\eec}\Le c(\uufm)
\sum_{i,j}\is{\varphi(\ssfm^{*}_{j}\ssfm_{i})f_{i}}{f_{j}}_{\eec},\quad 
(s_{k})_{k}\subset\ssf\;. (f_{k})_{k}\subset\eec;
\tag{$\alpha$}
\end{equation}
\begin{equation}
\|\varphi(\uufm)f\|_{\eec}\Le d(f)s(\uufm),\quad s \text{ submultiplicative.}\tag{$\beta$}
\end{equation}
If $\varphi$ is positive definite then ($\alpha$) and ($\beta$)  are among the equivalent conditions for boundedness of $\varPhi(\uufm)$, cf. \cite{pams} or \cite{ark}, also \cite{murphy} for much more general context. Condition ($\alpha$) is the original boundedness condition appearing in \cite{app}.

This and Remark \ref{1.01.01} contributes to Corollary \ref{2.10.01} bringing it closer Th\'eor\`eme Principal of \cite{app}.
\begin{cor}
\label{3.10.01}
If either of the conditions $(\alpha)$ and $(\beta)$ is satisfied the the operators $\varPhi(\uufm)$ in {\em Corollary \ref{2.10.01}} are bounded in $\ddc$, hence they extend to bounded operators in $\kkc$. 
\end{cor}

Theorem \ref{3.18.12} has now its counterpart as well.
\begin{cor}\label{4.10.01}
Under the circumstances of\, {\em Corollaries 	\ref{2.10.01}} and {\em \ref{3.10.01}} suppose $\varPhi$ is a minimal  dilation of $\varphi$. If $(\ssfm_{\alpha})_{\alpha}$ is net such that for any $\ttfm,\uufm\in\ssf$
\begin{equation*}
\is{\varphi(\ttfm\ssfm_{\alpha}\uufm)f}g_{\eec}\rightarrow\is{\varphi(\ttfm\ssfm\uufm)f}g_{\eec},\quad f,g\in\ffc.
\end{equation*}
then 
\begin{equation}\label{2bis.18.12}
\is{\varPhi(\ssfm_{\alpha})w}u_{\ddc}\rightarrow\is{\varPhi(\ssfm)w}u_{\ddc},\quad w,u\in \ddc.
\end{equation}
If either $\sup_{\alpha}c(\ssfm_{\alpha})<+\infty$ or, equivalently, $\sup_{\alpha}d(\ssfm_{\alpha})<+\infty$ then \eqref{2bis.18.12} can be replaced by
\begin{equation*}
\varPhi(\ssfm_{\alpha})\rightarrow\varPhi(\ssfm)
\end{equation*}
in the week topology of $\kkc$.
\end{cor}
Notice that  unlike in \cite{app} $\varphi(\ssfm)$'s are not supposed to be bounded operators. This comes {\em a posteriori} as a consequence of the imposed boundedness condition $(\alpha)$ or ($\beta$). So what we have got in this Subsection is a natural generalization of Th\'eor\`eme Principal of \cite{app}.

\begin{rem}\label{1.28.05}
Notice that \eqref{1.07.01}, in view of \eqref{1.10.01},  can be red as
\begin{equation*}
\is\varphi{\varphi_{\ttfm,g}}_{\ddc}=\is{\varphi(\ttfm^{*})}g_{\eec}=\is g{\varphi(\ttfm)}_{\eec}
\end{equation*}
for $\varphi=\sum_{i}\varphi_{\ssfm_{i},f_{i}}\in\ddc$ which is just a  kind of the reproducing kernel property. On the other hand, it has been know for long time in the case of scalar valued kernel, cf. \cite[pp. 37-38]{aron}, a completion of $\ddc$ can be chosen to be sill a space of scalar functions.
 This make it tempting to try realizing the completion within the space $\eec^{\ssf}$ of $\eec$--functions on $\ssf$.  Let us mention that if the operators $\varphi(s)\in{\bf L}(\ffc,\eec)$, $s\in\ssf$, are bounded then one can prove, adapting  that presented in \cite[pp. 5-7]{sza},  the completion of $\ddc$ can be achieve as the space of functions on $\ssf$ taking values in the completion of $\ffc$.

%
%
%
%
%
\end{rem}

\subsection{Back to Na\u{\i}mark dilations}\label{11.30.12}
Let us see what happens to Corollary \ref{2a.08.08} in the current environment. Suppose $\ssf$ is a semigroup defined as there considered with involution being the identity mapping and $\varphi\colon\ssf\mapsto L(\ffc,\eec)$ is a {\em positive operator valued measure} (recall $\ffc\subset\eec$, cf. Subsection \ref{9.30.12}). The latter means that  $\is{\varphi(\,\cdot)f}{f}_{\eec}$ is a scalar positive measure for every $f\in\ffc$. Adapting the argument used in \cite[pp. 30-31]{mlak} we get  $\sum_{i,j}\is{\varphi(\ssfm^{*}_{j}\uufm\ssfm_{i})f_{i}}{f_{j}}_{\eec}$ as a function of $\uufm$ is a positive scalar measure and therefore $\varphi $ is positive definite. Because $\sum_{i,j}\is{\varphi(\ssfm^{*}_{j}\uufm\ssfm_{i})f_{i}}{f_{j}}_{\eec}$ as a positive measure is increasing in $\uufm$ condition ($\alpha$) is automatically satisfied with $c(\uufm)=1$. This implies that all the operators $\varPhi(\uufm)$ are bounded (in fact are of the norm at most $1$).

Due to the specific nature of the $\ast$--semigroup $\ssf$ it comes from Theorem \ref{1.08.08}, conclusion (a), the operators $\varPhi(\uufm)$ are continuous idempotents in $\ddc$ therefore they extend to bounded idempotents in $\kkc$. Corollary \ref{2.10.01} tells us that they are selfadjoint. Thus the operators $\varPhi(\uufm)$, $\uufm\in\ssf$, are orthogonal projections. 

Furthermore, Corollary \ref{4.10.01} ensures $\varPhi$ to be a spectral measure. Making use of Remark \ref{1.01.01} we may state the following.

\begin{cor}
If $\varphi\colon\ssf\mapsto L(\ffc,\eec)$ is a positive operator valued measure such that 
\begin{equation}\label{3.11.06}
\is{\varphi(1)f}f_{\eec}\Le c\|f\|^{2}_{\eec}
\end{equation}
 then there is a Hilbert space $\kkc$ and a bounded $($of norm $\Le c^{1/2}$$)$ operator $\funk V \hhc \kkc$ 
 such that 
\begin{equation*}
\varphi(\uufm)=V^{*}\varPhi(\uufm)V,\quad \uufm\in\ssf.
\end{equation*} 
If instead of \eqref{3.11.06} one has 
$$
\is{\varphi(1)f}f_{\eec}=\|f\|^{2}_{\eec}
$$
then $V$ is an isometric imbedding and, consequently, $V^{*}$ becomes
 the orthogonal projection of $\kkc$ onto the closure of \,$\ddc$ in $\kkc$.
\end{cor}
This is a kind of Na\u{\i}mark's theorem \cite{naj} tailored to meet our more general needs, in particular if an operator valued measure defined by \eqref{2.16.12} is positive definite.
\vspace{9pt}
\subsubsection*{Acknowlegdement} The authors would like to express their thanks to the referee for the comments letting them remaster the paper.

  \bibliographystyle{amsplain}

    \end{document}